\documentclass[a4paper,12pt]{article}

\usepackage{amsmath,amssymb,amsthm}
\usepackage{graphicx,color,hyperref}
\usepackage{rawfonts}

\input prepictex
\input pictex
\input postpictex

\newcounter{mathitem}
\newenvironment{mathitem}
  {\begin{list}{$(\roman{mathitem})$}{
   \setcounter{mathitem}{0}
   \usecounter{mathitem}
   \setlength{\topsep}{0pt plus 2pt minus 0pt}
   \setlength{\parskip}{0pt plus 2pt minus 0pt}
   \setlength{\partopsep}{0pt plus 2pt minus 0pt}
   \setlength{\parsep}{0pt plus 2pt minus 0pt}
   \setlength{\leftmargin}{35pt}
   \setlength{\itemsep}{0pt plus 2pt minus 0pt}}}
 {\end{list}}

\newtheorem{theorem}{Theorem}
\newtheorem{definition}[theorem]{Definition}
\newtheorem{lemma}[theorem]{Lemma}

\newtheorem{corollary}[theorem]{Corollary}

\title{A Best Possible Result for the Square of a 2-Block to be Hamiltonian}
\author{Jan Ekstein\thanks{Department of Mathematics and European Centre of Excellence NTIS - New Technologies for the Information Society, Faculty of Applied Sciences, University of West Bohemia, Pilsen, Technick\'a 8, 306 14 Plze\v n, Czech Republic
        \newline e-mail: \texttt{ekstein@kma.zcu.cz}.}\and
        Herbert Fleischner\thanks{Institute of Logic and Computation, Algorithms and Complexity Group, Technical University of Vienna, Favoritenstrasse 9-11, 1040 Wien, Austria, EU  
        \newline e-mail: \texttt{fleischner@ac.tuwien.ac.at}.}}
\date{\today}

\begin{document}
\maketitle

\begin{abstract}
 It is shown that for any choice of four different vertices $x_1,...,x_4$ in a 2-block $G$ of order 
 $p>3$, there is a hamiltonian cycle in $G^2$ containing four different edges $x_iy_i$ of 
 $E(G)$ for certain vertices $y_i$, $i=1,2,3,4$. This result is best possible.
     
 {\bfseries Keywords}: square of graphs, hamiltonian cycles

 {\bfseries 2010 Mathematics Subject Classification:} 05C38,05C48
\end{abstract}

\section {Introduction}
As for standard terminology, we refer to the book by Bondy and Murty, \cite{Bon}, and to the papers quoted in the references.

The {\em square} of a graph $G$, denoted $G^2$, is the graph obtained from $G$ by joining any two nonadjacent vertices which have a common neighbor,  by an edge. Fairly recent development in hamiltonian graph theory has shown a resurgence of interest in hamiltonian cycles and paths in the square of 2-connected graphs (which we call 2-blocks for short). In particular, short proofs have been found for two results of the second author of the present paper, \cite{Geo}, \cite{Mut}. And more recently, in \cite{Als} the authors develop algorithms which are linear in $|E(G)|$ and produce a hamiltonian cycle, a hamiltonian path joining arbitrary vertices $u$ and $v$ respectively, in $G^2$. Moreover, they develop an algorithm running in $O(|V(G)|^2)$ time and producing cycles of arbitrary length from 3 to $|V(G)|$. 

Also very recently it was shown in \cite{FleEk} and \cite{FleChia} that a 2-block has the          $\mathcal{F}_{4}$ property; that is, given vertices $x_1,x_2,x_3,x_4$ in the 2-block $G$, there is a hamiltonian path in $G^2$ joining $x_1$ and $x_2$ and traversing distinct edges $x_3y_3$ and $x_4y_4$ of $G$ (see Theorem \ref{F_4}). The proof of this result is very long and is based on techniques developed by Fleischner in \cite{Fle}, \cite{Fle1}, \cite{Fle2} and by Fleischner and Hobbs in \cite{FleHob}. It remains to be shown whether one can find a much shorter proof of this result. However, this result will be of importance in the proof of the main result of the current paper.

We start with a definition.

\begin{definition}
 A graph $G$ is said to have the $\mathcal{H}_{k}$ property if for any given vertices $x_1,...,x_k$  
 there is a hamiltonian cycle in $G^2$ containing distinct edges $x_1y_1,...,x_ky_k$ of $G$.
\end{definition}

We note in passing that $G$ having the $\mathcal{F}_{4}$ property implies that $G$ has the 
$\mathcal{H}_{3}$ property; clearly, choose $x_1$, $x_2$, $x_3$ arbitrarily and a different $x_4$ adjacent to some $x_1$ for $i\in\{1,2,3\}$ in $G$, say $i=1$. A hamiltonian path in $G^2$ joining $x_1$ and $x_4$ and containing edges $x_2y_2$ and $x_3y_3$ of $G$ yields a hamiltonian cycle containing these two edges of $G$ and $x_1x_4$ which lies also in $G$.

The main result of this paper is the following.

\begin{theorem}
 \label{maintheorem}
 Given a 2-block $G$ on at least four vertices, then $G$ has the $\mathcal{H}_{4}$ property, and  
 there are 2-blocks of arbitrary order greater than 4 without the $\mathcal{H}_{5}$ property.
\end{theorem}

This theorem and the $\mathcal{F}_{4}$ property of 2-blocks are key to describe the most general block-cut vertex structure a graph $G$ may have in order to guarantee that $G^2$ is hamiltonian, hamiltonian connected, respectively. This will be done in follow-up papers.

Moreover Theorem \ref{maintheorem} gives the positive answer to Conjecture 5.4 stated in~\cite{Ekst} as an immediate corollary.

\begin{corollary}
 Let $G$ be a connected graph such that its block-cutvertex graph $\mbox{bc}(G)$ is homeomorphic to a star in which the center $c$ corresponds to a block $B_c$ of $G$. If $B_c$ contains at most 4 cutvertices, then $G^2$ is hamiltonian.
\end{corollary}

\section{Preliminaries}
However, before proving Theorem \ref{maintheorem} we mention several concepts and results which we need to make use of, and we prove a lemma.

A graph $G$ is an edge-critical block, if $\kappa(G)=2$ and $\kappa(G-e)=1$ for any edge $e$ of $G$ . Let $D(G)$ be the set of edges $uv$ where $d_{G}(u), d_{G}(v)\geq 3$. If $D(G)=\emptyset$, then every edge of $G$ is incident to a vertex of degree 2; we call such a graph a \emph{DT-graph}.

\begin{theorem}\emph{\textbf{\cite{Fle1}}}
 \label{The1}
 Let $G$ be an edge-critical block. Then exactly one of the following two statements is true:
 \begin{itemize}
  \item[1)] $G$ is a DT-block.
  \item[2)] There is an edge $f$ in $D(G)$ such that at least one of the endblocks of $G-f$ is a DT-block.
 \end{itemize}
\end{theorem}

The basic result about hamiltonicity of the square of a 2-block is given by the following theorem.

\begin{theorem}\emph{\textbf{\cite{Fle2}}}
\label{2-blockcycle}
Suppose $v$ and $w$ are two arbitrarily chosen vertices of a $2$-block $G$.
Then $G^2$ contains a hamiltonian cycle $C$ such that the edges of $C$ incident to $v$ are in $G$ 
and at least one of the edges of $C$ incident to $w$ is in $G$. Furthermore, if $v$ and $w$ are 
adjacent in $G$, then these are three different edges.  
\end{theorem}

Let $\mbox{bc}(G)$ denote the block-cutvertex graph of $G$. Blocks corresponding to leaves of $\mbox{bc}(G)$ are called \emph{endblocks}. Note that a block in a graph $G$ is either a 2-block or a bridge of $G$. The graph $G$ is called \emph{blockchain} if $\mbox{bc}(G)$ is a path. Let $G$ be a blockchain. We denote its blocks $B_{1}, B_{2},...,B_{k}$ and cutvertices $c_{1},c_{2},...,c_{k-1}$ such that $c_{i}\in V(B_{i})\cap V(B_{i+1})$, for $i=1,2,...,k-1$. A blockchain $G$ is called \emph{trivial}, if $E(\mbox{bc}(G))=\emptyset$, otherwise it is called  \emph{non-trivial}. Note that only $B_{1}$ and $B_{k}$ are endblocks of a non-trivial blockchain $G$. \emph{An inner block} is a block of $G$ containing exactly 2 cutvertices. \emph{An inner vertex} is a vertex in $G$ which is not a cutvertex of $G$.

The first author proved in \cite{Ekst} the following theorem dealing hamiltonicity of the square of a blockchain graph.

\begin{theorem}\emph{\textbf{\cite{Ekst}}}
 \label{blockchaincycle}
 Let $G$ be a blockchain and let $u_{1}$, $u_{2}$ be arbitrary inner vertices which are contained in different endblocks of~$G$.\\
 Then $G^{2}$ contains a hamiltonian cycle $C$ such that, for $i=1,2$,
  \begin{mathitem}
     \item[$\bullet$] if $u_{i}$ is contained in a 2-block, then both edges of $C$ incident with 
                      $u_{i}$ are in $G$, and   
     \item[$\bullet$] if $u_{i}$ is not contained in a 2-block, then exactly one edge of $C$ 
                      incident with $u_{i}$ is in $G$.
  \end{mathitem}
\end{theorem}

Let $G$ be a connected graph. By a \emph{$uv$-path} we mean a path from $u$ to $v$ in $G$. If a $uv$-path is hamiltonian, we call it a \emph{$uv$-hamiltonian path}. Let $A=\{x_{1},x_{2},...,x_{k}\}$ be a set of $k$ ($\geq 3$) distinct vertices in $G$. An $x_{1}x_{2}$-hamiltonian path
in $G^{2}$ which contains $k-2$ distinct edges $x_{i}y_{i}\in E(G), i=3,...,k$, is said to be $\mathcal{F}_{k}$. A graph $G$ is said to have the $\mathcal{F}_{k}$ property if, for any set $A=\{x_{1},x_{2},...,x_{k}\}\subseteq V(G)$, there is an $\mathcal{F}_{k}$ $x_{1}x_{2}$-hamiltonian path in $G^{2}$.

\begin{theorem}\emph{\textbf{\cite{FleChia}}}
 \label{F_4}
  Let $G$ be a 2-block. Then $G$ has the $\mathcal{F}_{4}$ property. 
\end{theorem}

A graph $G$ is said to have the strong $\mathcal{F}_{3}$ property if, for any set of 3 vertices $\{x_{1},x_{2},x_{3}\}$ in $G$, there is an 
$x_{1}x_{2}$-hamiltonian path in $G^{2}$ containing distinct edges $x_{3}z_{3},x_{i}z_{i}\in E(G)$ for a given $i\in\{1,2\}$. Such an 
$x_{1}x_{2}$-hamiltonian path in $G^{2}$ is called a strong $\mathcal{F}_{3}$ $x_{1}x_{2}$-hamiltonian path.

\begin{theorem}\emph{\textbf{\cite{FleChia}}}
 \label{strongF_3}
  Every 2-block has the strong $\mathcal{F}_{3}$ property.
\end{theorem}

The following lemma is frequently used in the proofs below.

\begin{lemma}
 \label{blockchainpath}
 Let $G$ be a non-trivial blockchain. We choose
 \begin{itemize}
  \item $c_{0}\in V(B_{1})$, $c_{k}\in V(B_{k})$ which are not cutvertices;
  \item $u_{i}\in V(B_{i})$ (if any) which is not a cutvertex and $v_{i}\in V(B_{i})$ such that 
        $u_{i}\neq v_{i}$, $u_{1}\neq c_{0}$ and $u_{k}\neq c_{k}$, for $i=1,2,...k$.
\end{itemize} 
 Then $G^{2}$ contains a $c_{0}c_{k}$-hamiltonian path $P$ such that there exist distinct edges $u_{i}u'_{i}$ $v_{i}v'_{i}\in E(B_{i})\cap E(P)$ (if $u_{i}$ exists), $i=1,2,...,k$.
\end{lemma}

\begin{proof}
 If $B_{i}$ is 2-connected, then let $P_{i}$ be an $\mathcal{F}_{4}$ $c_{i-1}c_{i}$-hamiltonian path in $B_{i}^{2}$ containing 2 distinct edges $u_{i}u'_{i},v_{i}v'_{i}\in E(B_{i})$ for $v_{i}\notin \{c_{i-1},c_{i}\}$ by Theorem~\ref{F_4}; and let $P_i$ be a strong $\mathcal{F}_{3}$ $c_{i-1}c_{i}$-hamiltonian path in $B_{i}^{2}$ containing 2 distinct edges $u_{i}u'_{i},v_{i}v'_{i}\in E(B_{i})$ for $v_{i}\in \{c_{i-1},c_{i}\}$ by Theorem~\ref{strongF_3}, respectively. 
 
 If $B_{i}=c_{i-1}c_{i}$, then we set $P_{i}=B_{i}$. Note that in this case $u_{i}$ does not exist and $v_{i}\in \{c_{i-1},c_{i}\}$. 
 
 Then $P=\cup_{i=1}^{k}P_{i}$ is a $c_{0}c_{k}$-hamiltonian path in $G^{2}$ as required. 
\end{proof}

The concept of EPS-graphs plays a central role in proofs of hamiltonicity in the square of a $DT$-graph (see \cite{Fle}). We use this concept also in one part of the proof of Theorem \ref{maintheorem}. Let $G$ be a graph. An \emph{EPS-graph} is a spanning connected subgraph $S$ of $G$ which is the edge-disjoint union of an eulerian graph $E$ (which may be disconnected) and a linear forest $P$. For $S=E\cup P$, let $d_{E}(v)$, $d_{P}(v)$ denote the degree of $v$ in $E$, $P$, respectively. 

Fleischner and Hobbs introduced in \cite{FleHob} the concept of $W$-soundness of a cycle. Let $W$ be a set of vertices of $G$. A cycle $K$ is called 
\emph{$W$-maximal} if $|V(K')\cap W|\leq|V(K)\cap W|$ for any cycle $K'$ of $G$. Let $K$ be a cycle of $G$ and let $W$ be a set of vertices of $G$. 
A blockchain $P$ of $G-K$ is \emph{a $W$-separated $K$-to-$K$ blockchain based on vertex $x$} if a vertex of $W$ is a cut vertex of $P$, both endblocks $B$ and $B'$ of $P$ include vertices of $K$, $V(B)\cap V(K)=\{x\}$, no vertex of $K$ is a cutvertex of $P$, and $(V(P)\cap V(K))-\{x\}\subseteq V(B').$ For a given path $p=v_1,v_2,...,v_{n-1},v_n$ we let $F(p)=v_1$, $L(p)=v_n$.

\begin{definition}
 \label{sound}
 A cycle $K$ in $G$ is \emph{$W$-sound} if it is $W$-maximal, $|W|=5$ and the following hold:
\begin{itemize}
 \item[(1)] $|V(K)\cap W|\geq 4$; or
 \item[(2)] $|V(K)\cap W|=3$ and the following situation does not prevail; there are two $W$-separated $K$-to-$K$ blockchains $P$ and $Q$ of $G-K$ based on
            a vertex $w$ of $W$ such that $V(P)\cap V(Q)=\{w\}$ and if $p$ is a shortest path in $P$ from $w$ to a vertex of $K$ different from $w$ and $q$
            is the same for $Q$, then there is a subsequence $w,w',L(p),L(q),w'',w$ of $K$ where $w'$ and $w''$ are in $W-\{w\}$; or
 \item[(3)] $|V(K)\cap W|=2$ and the following situation does not prevail; there are three $W$-separated $K$-to-$K$ blockchains $P_1,P_2$ and $P_3$ of $G-K$
            based on a single vertex $a$ of $V(K)-W$, such that $V(P_i)\cap V(P_j)=\{a\}$ whenever $i$ and $j$ are distinct elements of $\{1,2,3\}$, and if 
            $p_i$ is a shortest path in $P_i$ from $a$ to a vertex of $K$ different from $a$ for each $i\in \{1,2,3\}$, then there is a subsequence
            $a,w',L(p_1),L(p_2),L(p_3),w'',a$ of $K$ where $\{w',w''\}=V(K)\cap W$.
\end{itemize}
\end{definition}

We observe that Definition \ref{sound} is basically the content of Lemma 1 in \cite{FleHob}. That is, said lemma guarantees that for every choice $W\subseteq V(G)$ with $|W|=5$ in a 2-block $G$ of order at least 5, there is a $W$-sound cycle in $G$.

\begin{theorem}\emph{\textbf{\cite{FleHob}}}
 \label{The2}
 Let $G$ be a 2-block and $W$ a set of five distinct vertices in $G$, and let $K$ be a $W$-sound cycle in $G$. Then there is an EPS-graph 
 $S=E\cup P$ of $G$ such that $K\subseteq E$ and $d_{P}(w)\leq 1$ for every $w\in W$.
\end{theorem} 



\section{Proof of Theorem \ref{maintheorem}}
 
\begin{proof}
 First we prove that $G$ has the $\mathcal{H}_{4}$ property. We proceed by contradiction supposing that $|V(G)|+|E(G)|$ is minimal. It follows that $G$ is an edge-critical block and in particular $|V(G)|\geq 5$. We distinguish cases by the number of edges in $D(G)$. The reader is advised to draw figures where he/she deems it necessary to follow our case distinctions.
 
 \noindent 
 \emph{Case 1.} $|D(G)|>0.$
  By Theorem~\ref{The1}, let $f=x'x\in D(G)$ be an edge such that $d_{G}(x'), d_{G}(x)\geq 3$.
  Then $G-f$ is a blockchain and both endblocks $B',B$ of $G-f$ are 2-blocks. Set 
  $X=\{x_1,x_2,x_3,x_4\}$. Without loss of generality assume that $|X\cap(V(B)-y)|\leq 2$ 
  (otherwise we consider $B'$ instead of $B$); i.e., at most $x_{1},x_2\in V(B)-y$, say, where 
  $x, y\in V(B)$ and $y$ is a cutvertex of $G-f$. We distinguish the following 3 subcases. 

 \medskip  
  
 \emph{Subcase 1.1}:  $|X\cap(V(B)-y)|=2$; i.e., $x_1, x_2\in V(B)-y.$
 
  Then $B^2$ has an $xy$-hamiltonian path $P_1$ containing different edges $x_1y_1$, $x_2y_2$ of 
  $E(G)$ for certain $y_1,y_2$ by Theorem \ref{F_4} or by Theorem \ref{strongF_3} if $x_1=x$ or 
  $x_2=x$; and $(G-B)^2$ has an $xy$-hamiltonian path $P_2$ containing different edges 
  $x_3y_3,x_4y_4$ of $E(G)$ for certain $y_3,y_4$ by Lemma \ref{blockchainpath}. Now $P_1 \cup P_2$ 
  is a required hamiltonian cycle in $G^2$, a contradiction. Note that $x_3,x_4\in V(B')-y'$ 
  where $y'\in V(B')$ is a cutvertex of $G-f$, otherwise we can use $B'$ instead of $B$ and 
  $x_3$ or $x_4$ instead of $x_1$ or $x_2$ (see \emph{Subcase 1.2} or \emph{Subcase 1.3} below).    
   
 \medskip
 
 \emph{Subcase 1.2}: $|X\cap(V(B)-y)|=1$; i.e., $x_1 \in V(B)-y$ and $x_2 \notin V(B)-y.$
 
 \emph{(1.2.1)} Assume that $x_2,x_3,x_4$ are not inner vertices of $G$ in the same block of $G-B$.
 We proceed very similar as in \emph{Subcase 1.1}; we use only the strong $\mathcal{F}_{3}$ 
 property in $B$, and $G-B$ is a non-trivial blockchain. Hence we can apply Lemma
 \ref{blockchainpath} except if $x=x_1$, some $x_i=y$ for $i\in \{2,3,4\}$, say $i=2$, and 
 $x_3,x_4$ are inner vertices in the same endblock of $G-B$ which also contains $x_2$. 
 
 If $x=x_1$, $x_2=y$, and $x_3,x_4$ are inner vertices in the same endblock of $G-B$ which also 
 contains $x_2$, then $B^2$ has an $x_2x_1$-hamiltonian path $P_1$ containing different edges 
 $x_2y_2$, $uv$ of $E(G)$ for certain $y_2,u,v$ by Theorem~\ref{strongF_3}, and $(G-B)^2$ has an 
 $x_2x_1$-hamiltonian path $P_2$ containing different edges $x_1x',x_3y_3,x_4y_4$ of $E(G)$ for 
 certain $y_3,y_4$ by Lemma \ref{blockchainpath}. Again, $P_1 \cup P_2$ is a required hamiltonian 
 cycle in $G^2$, a contradiction.
 
 \emph{(1.2.2)} Assume that $x_2,x_3,x_4$ are inner vertices of $G$ in the same block $B^*$ 
 of $G-B$.
 
 Clearly, $B^2$ contains a hamiltonian cycle $H_{B}$ containing 3 different edges $y'y,x'_1x_1,x''x$ of $E(B)$ for certain vertices $y',x'_1,x''$ by Theorem \ref{F_4} (starting with a corresponding $\mathcal{F}_{4}$ $x''x$-hamiltonian path in $B^{2}$) if $x\neq x_1$, and $y'y,x'_1x,x''x$ of $E(B)$ for certain vertices $y',x'_1,x''$ by Theorem \ref{2-blockcycle} if $x=x_1$.

 Let $G_1$ be the component of $G-B^*-xx'$ containing $B$ and $y^*=V(B^*)\cap V(G_1)$. Note that 
 $G_1$ is a trivial or non-trivial blockchain.
 
 (a) If $y^*=y$, then $G_1=B$ and we set $H_{G_1}=H_B$ (see above). 
 
 (b) If $y^*\neq y$, then either $G_1-B=y^*y$ or $(G_1-B)^2$ contains a hamitonian cycle $C$ containing edges $y^*_1y^*$, $y''y$ of $E(G_1-B)$ for certain $y^*_1,y''$ by applying Theorem \ref{2-blockcycle} or Theorem \ref{blockchaincycle}. 
 
 Now we set $$H_{G_1}=(H_B-y'y)\cup y'y^*$$ and $y_1^*=y$ if $G_1-B=y^*y$; and 
 $$H_{G_1}=(H_B\cup C-\{y'y,y''y\})\cup y'y''$$ if $G_1-B\neq y^*y$. 
 
 Note that the edge $y_1^*y^*\in E(G_1)$ is contained in $H_{G_1}$ in both cases.  
 
 \medskip

 Clearly, $|V(B^*)|+|E(B^*)|<|V(G)|+|E(G)|$. Hence $(B^*)^2$ contains a hamiltonian cycle $H_{B^*}$ containing four different edges $y_2^*y^*,x_2x'_2,x_3x'_3,x_4x'_4$ of $E(B^*)$ for certain vertices $y_2^*,x'_i$, $i=2,3,4$.
 
 \medskip
 
 Let $z\in V(B^*)$ be the cutvertex of $G-x'x$ different from $y^*$. 
 
 (A) $x'=z$. Then $$(H_{G_1}\cup H_{B^*}-\{y_2^*y^*,y_1^*y^*\})\cup\{y_1^*y_2^*\}$$ is a required hamiltonian cycle in $G^2$ containing four different edges $x_ix'_i,$ of $E(G)$, $i=1,2,3,4$, a contradiction.
  
 (B) $x'\neq z$
 
 If $d_{G-B^*}(z)=1$, then we set $G_2=G-G_1-B^*-z_1z$ where $z_1$ is the unique neighbour of 
 $z$ in $G-B^*$; otherwise we set $G_2=G-G_1-B^*$. Note that $G_2$ is a trivial or non-trivial 
 blockchain and $G_2=x'x$ is not possible because of $d_G(x')>2$.

We apply Theorem \ref{blockchaincycle} such that either $(G_2)^2$ contains a hamitonian cycle $H_{G_2}$ with $x'x\in E(H_{G_2})$ if $z\notin V(G_2)$, or $(G_2)^2$ contains a hamitonian cycle $H$ containing the edge $x'x$ and different edges $z_1z,z_2z$ of $G_1$ for certain $z_1,z_2$ if $z\in V(G_2)$. In the latter case we set $H_{G_2}=(H-\{z_1z,z_2z\})\cup z_1z_2$. Then $$(H_{G_1}\cup H_{G_2}\cup H_{B^*}-\{y_2^*y^*,y_1^*y^*,x'x,x''x\})\cup\{y_1^*y_2^*,x''x'\}$$ is again a hamiltonian cycle in $G^2$ containing four different edges $x_ix'_i$ of $E(G)$, $i=1,2,3,4$, a contradiction.
 
\medskip 
  
 \emph{Subcase 1.3}: $|X\cap(V(B)-y)|=0$; i.e., $x_1, x_2\notin V(B)-y.$
 
 Let $G_1$ be a graph which arises from $G$ by replacing $B$ with a path $p$ of length 3, 
 say $p=x,a,b,y$. Then $|V(G_1)|+|E(G_1)|<|V(G)|+|E(G)|$ since $B$ is not a triangle because $G$ 
 is edge-critical. Hence $(G_1)^2$ contains a hamiltonian cycle $H_1$ containing four different 
 edges $x_iy_i$ of $E(G_1)$ for certain vertices $y_i$, $i=1,2,3,4$, and as many edges as possible 
 of $G_1$.
 
 In the following we shall proceed in a manner very similar to the proof that the square of a 2-block is hamiltonian, \cite{Fle1}. However, in order to avoid total dependence of the reader on the knowledge or study of \cite{Fle1}, we shall describe and partially repeat the procedure employed in that paper. In particular, we shall quote the cases with the numbering of \cite{Fle1}.
 
 This yields the consideration of 13 cases how the hamiltonian cycle $H_1$ traverses vertices of the path $p$. As in \cite{Fle1}, Cases 3, Case 4, Case 12, and Case 13 are contradictory to the maximality of the number of edges of $G_1$ belonging to $H_1$; and Case 6 can be reduced to Case 10, Case 8 to Case 7, Case 10 to Case 9 and Case 11 to Case 5. Note that by the reductions we preserve the existence of the edges $x_iy_i$ even if $x_i\in\{x',y\}$ for $i\in\{1,2,3,4\}$. 
 
 The remaining 5 cases are (using the labeling of vertices $x',x,a,b,y$ instead of $x,w,a,b,v$ in \cite{Fle1}): 
 
 \emph{Case 1.} $H_1=...,x,a,b,y,...$ 
 
 \emph{Case 2.} $H_1=...,x,a,b,y',...$  
  
 \emph{Case 5.} $H_1=...,x',a,b,x,...$  
   
 \emph{Case 7.} $H_1=...,x',a,y,...,y',b,x$  
    
 \emph{Case 9.} $H_1=...,x',a,y,b,x...$;
 
\noindent
 and $y'y$ is an edge of $G$.    
 
 In order to extend $H_1$ to $H$ in $G^2$ in these five cases with $H$ having the required property, one can proceed in the same way as it has been done in~\cite{Fle1}. However, we deem it necessary to show explicitly that no problems arise under the stronger condition of this theorem (similarly as in \cite{Fle2}).

 \emph{Case 1.} By Theorem \ref{strongF_3}, $B^2$ has an $xy$-hamiltonian path $P$ starting with an edge $yy^*$ of $E(B)$ and containing an edge $uv$ of $B$ for certain vertices $u,v$. Replace in $H_1$ the path $p$ with a hamiltonian path $P$ and we get a hamiltonian cycle $H$ as required.
 
 \emph{Case 2.} Take $P$ as in Case 1 and replace in $H_1$ the path $x,a,b,y'$ with $(P-yy^*)\cup y'y^*$ and again we get a hamiltonian cycle $H$ as required. Note that $H$ contains all edges of $G$ belonging to $H_1$.
 
 \emph{Case 5.} By Theorem \ref{2-blockcycle}, $B^2$ contains a hamiltonian cycle $H_B$ such that both edges of $H_B$ incident to $y$ (say $yy^*,yy^{**}$) are in $B$ and at least one of the edges of $H_B$ incident to $x$ (say $xx^*$) is in $B$. We set $$H^*=(H_B-\{yy^*,yy^{**}\})\cup y^*y^{**}$$ which does not contain $y$, and replace in $H_1$ the path $x',a,b,x$ with $(H^*-xx^*)\cup x'x^*$, thus obtaining a hamiltonian cycle $H$ in $G^2$ which has the same behavior in all vertices of $G_1-\{a,b\}\subset G$ as $H_1$.

 \emph{Case 7.} Take $H_B$ as in Case 5 and replace in $H_1$ the path $x',a,y$ with the path $P_1\cup x^*x'$ where $P_1\subset H_B$ is the path from $y$ to $x^*$ and does not contain $x$; and 
replace in $H_1$ the path $y',b,x$ with the path $P_2\cup y't$ where $t\in \{y^*,y^{**}\}$ and $P_2\subset H_B$ is the path from $x$ to $t$ and does not contain any of $y$, $x^*$. Again we get a hamiltonian cycle $H$ as required.

 \emph{Case 9.} Take $H_B$ as in Case 5 and replace in $H_1$ the path $x',a,y,b,x$ with $(H_B-xx^*)\cup x'x^*$, thus obtaining a hamiltonian cycle $H$ in $G^2$ which has the same behavior in all vertices of $G_1-\{a,b,y\}\subset G$ as $H_1$ and both edges of $H$ incident to $y$ are in $G$.
 
 In all cases we obtained a hamiltonian cycle $H$ in $G^2$ containing four different edges $x_ix'_i,$ of $E(G)$ (in most cases we have $x'_i=y_i$; see the first paragraph of this subcase 1.3), $i=1,2,3,4$, a contradiction.

 \medskip  
  
 \noindent
 \emph{Case 2.} $|D(G)|=0.$ That is, $G$ is a $DT$-graph.
  
  \noindent
  a) Suppose $N(x_{i})\subseteq V_{2}(G)$ for every $i=1,2,3,4$.
  
 Set $W'=\{x_1,x_2,x_3,x_4\}$ and let $K$ be a $W'$-maximal cycle in $G$. Observe that $|V(K)|\geq 4$ since an edge-critical block on at least 4 vertices cannot contain a triangle. 
 
 If $|W'\cap V(K)|=4$, then we choose $x_5$ arbitrary in $V(G)-W'$. If $|W'\cap V(K)|=3$, then we choose $x_5$ arbitrary in $V(K)-W'$. If $|W'\cap V(K)|=2$, then we choose an arbitrary 2-valent vertex $x_5$ in $V(K)-W'$ which exists because all neighbours of $x_i$ are 2-valent.
 
 We set $W=W'\cup\{x_5\}$. Then $K$ is $W$-sound in $G$ unless $|W\cap V(K)|=3$ and forbidden situation (2) in Definition \ref{sound} arises. That is, without loss of generality $x_1,x_2\in V(K)$ and there exist $W$-separated $K$-to-$K$ blockchains $P$, $Q$ based on $x_i$, $i\in \{1,2\}$, $P\cap Q=x_i$, and paths $p,q$ in $P,Q$, respectively, such that there is a subsequence $x_i,w',L(p),L(q),w'',x_i$, where $\{w',w''\}=\{x_{3-i},x_5\}$ and $x_3,x_4\in V(p)\cup V(q)$. Then there is a cycle $K'$ containing $x_i,x_3,x_4$, a contradiction to the $W'$-maximality of $K$.
  
 By Theorem \ref{The2}, $G$ contains an EPS-graph $S=E\cup P$ such that $K\subseteq E$ and $d_{P}(w)\leq 1$ for every $w\in W$. If there is no adjacent pair $x_{i},x_{j}$ for $i,j\in \{1,2,3,4\}$, we use $S$ and an algorithm in \cite{Fle} to obtain a hamiltonian cycle in $G^{2}$ with the required properties, a contradiction. However, if there is an adjacent pair, say $x_1,x_2$, then $d_G(x_1)=d_G(x_2)=2$ and $d_P(x_1)=d_P(x_2)=0$ and we can proceed with the cycle $K$ containing $x_1,x_2,x_3$ to obtain a required hamiltonian cycle in $G^2$ as before, a contradiction.
  
  \noindent
  b) Without loss of generality suppose that $N(x_{4})\nsubseteq V_{2}(G)$. 
 
  Hence $\mbox{deg}_{G}(x_{4})=2$. Let $P_{4}=y_{4}x_{4}z_{1}...z_{k}$ be a unique path in $G$ such that $d_{G}(y_{4})>2$, $d_{G}(z_{k})>2$ and 
  $d_{G}(z_{i})=2$, for $i=1,2,...,k-1$. We set $G^-=G-\{x_{4},z_{1},...,z_{k-1}\}$, where $z_0=x_4$ if $k=1$.
  
  b1) Assume that $G^-$ is 2-connected.
  
  If $x_i\in V(G^-)-\{y_4,z_k\}$ for $i=1,2,3$, then $|V(G)|+|E(G)|>|V(G^-)|+|E(G^-)|$ and hence $(G^-)^{2}$ has a hamiltonian cycle $H^-$ containing  different edges $x_{i}y_{i},z_{k}w_{4}\in E(G)$, $i=1,2,3$. It is easy to see that we can extend $H^-$ to a hamiltonian cycle $H$ in $G^{2}$ such that $H$ contains edges $x_{i}y_{i}$, $x_{4}z_{1}$, for $i=1,2,3$, a contradiction.   
  
  Suppose $x_3\notin V(G^-)-\{y_4,z_k\}$. If $\{x_1,x_2,x_3\}\cap\{y_4,z_k\}\neq\emptyset\}$, then without loss of generality $x_3\in\{y_4,z_k\}$. By Theorem~\ref{F_4} or Theorem~\ref{strongF_3}, $(G^-)^{2}$ contains a $y_{4}z_{k}$-hamiltonian path $P^-$ and $P^-$ contains distinct edges $x_{i}y_{i}$ of $G$ if $x_i\in V(G^-)$ for $i=1,2$. Then $P^{-}\cup P_{4}$ is a hamiltonian cycle in $G^{2}$ with the required properties, a contradiction.
 
  b2) Assume that $G^-$ is not 2-connected.
  
  \noindent
  Then $G^-$ is a non-trivial blockchain with $y_{4},z_{k}$ in distinct endblocks and $y_{4},z_{k}$ are not cutvertices. 
  
  Assume not all $x_{1},x_{2},x_{3}$ are inner vertices in the same block. Then we apply Lemma~\ref{blockchainpath} to get a $y_{4}z_{k}$-hamiltonian path $P^-$ in $(G^-)^{2}$ with distinct edges $x_{i}y_{i}\in E(G^-)$, $i=1,2,3$. Note than $x_{i}$ could be $y_{4}$ or $z_{k}$. Then again 
$P^-\cup P_{4}$ is a hamiltonian cycle in $G^{2}$ with the required properties, a contradiction.  
  
  Now assume that $x_{1},x_{2},x_{3}$ are inner vertices in the same block $B$. Then there exists an end block $B^*$ of $G^-$ such that $x_{i}\notin V(B^*)$, $i=1,2,3$. A graph $G'$ arises from $G$ by the replacement of $B^*$ by a path $p$ of length 3. Hence $|V(G)|+|E(G)|>|V(G')|+|E(G')|$ and we denote by $H'$ a hamiltonian cycle in $(G')^{2}$ containing edges $x_{i}w_{i}$, $i=1,2,3,4$, and as many edges of $G'$ as possible.
  
 We proceed in the same manner as in Subcase 1.3 (note that in this case none of $x_i$, $i=1,2,3,4$, is on $p$) to get a hamiltonian cycle in $G^{2}$ with required properties, a contradiction. 
 
 \bigskip

 Finally we want to show that Theorem \ref{maintheorem} is best possible, i.e., we construct an infinite family of graphs which do not satisfy the $\mathcal{H}_{5}$ property. For this purpose start with an arbitrary 2-block $G$ and fix different vertices $x_1, x_2\in V(G)$. 
 
 Define $$H=G\cup \{y_1,y_2,...,y_t;t\geq 3\}\cup \{x_iy_j: 1\leq i\leq 2,1\leq j\leq t\},$$

where $\{y_1,...,y_t\}\cap V(G)=\emptyset$.

Then $H$ is a 2-block. However, $H$ does not have the $\mathcal{H}_{5}$ property: indeed, there is  no hamiltonian cycle $C$ in $H^2$ containing edges of $H$ incident to $x_1,x_2,y_1,y_2,y_3$ because of the neighbours of $y_1,y_2,y_3$ in $H$ are $x_1$ and $x_2$ only; that is $x_1$ or $x_2$ would be  incident to three edges of $C\cap H$, which is impossible.
\end{proof}
  
\section{Conclusion}
We introduced the concept of the $\mathcal{H}_{k}$ property and proved that every 2-block has the $\mathcal{H}_{4}$ property but not the $\mathcal{H}_{5}$ property in general. Similarly in \cite{FleChia} it is proved that every 2-block has the $\mathcal{F}_{4}$ property but not the $\mathcal{F}_{5}$ property in general. Moreover, a 2-block $G$ having the $\mathcal{F}_{k}$ property implies that $G$ has the $\mathcal{H}_{k-1}$ property for $k=3,4,...$. Hence we conclude that Theorem \ref{maintheorem} and Theorem \ref{F_4} are best possible with respect to hamiltonicity and hamiltonian connectedness in the square of a 2-block.

\medskip

\noindent 
{\bf Acknowledgements}.

\noindent
This publication was supported by the project LO1506 of the Czech Ministry of Education, Youth and Sports, and by FWF project P27615-N25.

\end{document}